\documentclass[a4paper,12pt]{amsart}
\usepackage[utf8]{inputenc}
\usepackage{hyperref}
\usepackage{fullpage}
\usepackage{amsmath, amssymb}
\usepackage[alphabetic,nobysame, initials]{amsrefs}
\usepackage{graphicx}

\newtheorem{thm}{Theorem}[section]
\newtheorem{lem}[thm]{Lemma}

\theoremstyle{definition}
\newtheorem{defn}[thm]{Definition}

\renewcommand{\Re}{\mathbb R}
\renewcommand{\epsilon}{\varepsilon}
\newcommand{\Red}{\Re^d}

\newcommand{\FF}{\mathcal F}

\newcommand{\st}{\; : \; }
\renewcommand{\phi}{\varphi}
\newcommand{\vol}[1]{\operatorname{vol}\left(#1\right)}

\newcommand{\bd}{\operatorname{bd}}

\newcommand{\conv}{\operatorname{conv}}

\newcommand{\prob}{\operatorname{Prob}}

\newcommand{\noshow}[1]{}

\title{Approximating a convex body by a polytope using the epsilon-net theorem}
\author[M. Nasz\'odi]{M\'arton Nasz\'odi}
\address{
M\'arton Nasz\'odi,
ELTE, Dept. of Geometry,
Lorand E\"otv\"os University,
P\'azm\'any P\'eter S\'et\'any 1/C
Budapest, Hungary 1117
}

\email{marton.naszodi@math.elte.hu}
\begin{document}
\begin{abstract}
Giving a joint generalization of a result of Brazitikos, Chasapis and Hioni and 
results of Giannopoulos and Milman, we prove that roughly 
$\left\lceil \frac{d}{(1-\vartheta)^d}\ln\frac{1}{(1-\vartheta)^d} \right\rceil$
points chosen uniformly and independently from a centered convex body $K$ in ${\mathbb R}^d$ yield 
a polytope $P$ for which 
$\vartheta K\subseteq P\subseteq K$ holds with large probability. The proof is 
simple, and relies on a combinatorial tool, the $\varepsilon$-net theorem.
\end{abstract}

\maketitle

\section{Introduction}

A convex body (i.e., a compact convex set with non-empty interior ) in $\Red$ 
is called \emph{centered}, if its center of mass is the origin.

We study the following problem. Given a convex body $K$ in $\Red$, a 
positive integer $t\geq d+1$, and $\delta,\vartheta\in(0,1)$. We want to show 
that under some assumptions on the parameters $d,t,\delta,\vartheta$ (and 
without assumptions on $K$), the convex hull of $t$ randomly, uniformly and 
independently chosen points of $K$ contains $\vartheta K$ with probability at 
least $1-\delta$.

The main result of \cite{BCH16} concerns the case of very rough approximation, 
that is, where the number $t$ of chosen points 
is linear in the dimension $d$. It states that the convex hull of 
$t=\alpha d$ random points in a centered convex body $K$ is a polytope $P$ 
which satisfies $\frac{c_1}{d}K\subseteq P$, with probability 
$1-\delta=1-e^{-c_2 d}$, where $c_1,c_2>0$ and $\alpha>1$ are absolute 
constants.

Our first result is a slightly stronger version of this statements, where the 
three constants are made explicit.
\begin{thm}\label{thm:brazitikos}
Let $K$ be a centered convex body in $\Red$. 
Choose $t=500d$ points $X_1,\ldots,X_t$ of $K$ randomly, independently and 
uniformly. Then 
\begin{equation*} 
\frac{1}{d} K\subseteq \conv \{X_1,\ldots,X_t\}\subseteq K.
\end{equation*}
with probability at least $1-1/e^d$.
\end{thm}

Another instance of our general problem is Theorem~5.2 of \cite{GM00}, which 
concerns fine approximation, that is, where the number $t$ of chosen points 
is exponential in the dimension $d$. It states that
for any $\delta,\gamma\in(0,1)$, if we choose $t=e^{\gamma d}$ 
random points in any centered convex body $K$ in $\Red$, then the polytope $P$ 
thus obtained satisfies $c(\delta)\gamma K\subseteq P$, with probability 
$1-\delta$. 

The same argument yields Proposition~5.3 of \cite{GM00}, according 
to which for any $\delta,\vartheta\in(0,1)$, if we choose 
$t=c(\delta)\left(\frac{c}{1-\vartheta} \right)^d$ random points in any 
centered convex body $K$ in $\Red$, then the polytope $P$ thus obtained 
satisfies $\vartheta K\subseteq P$, with probability $1-\delta$. 

Our main result is the following.
 
\begin{thm}\label{thm:approxgeneral}
Let $\delta, \vartheta\in (0,1)$, and let $K$ be a centered convex body in 
$\Red$. Let
\[
 t:=\left\lceil
 C\frac{(d+1)e}{(1-\vartheta)^d}\ln\frac{e}{(1-\vartheta)^d}
 \right\rceil,
\]
where $C\geq 2$ is such that
\begin{equation*}
C^2\left(\frac{(1-\vartheta)^d}{e}\right)^{C-2}\leq 
\frac{(\delta/4)^{1/(d+1)}}{e^3}.
\end{equation*}
Choose $t$ points $X_1,\ldots,X_t$ of $K$ randomly, independently and 
uniformly. 
Then
\begin{equation*} 
\vartheta K\subseteq \conv \{X_1,\ldots,X_t\}\subseteq K
\end{equation*}
with probability at least $1-\delta$.
\end{thm}

By substituting $\vartheta=\frac{1}{d}, \delta=e^{-d-1}, C=7$,
we obtain Theorem~\ref{thm:brazitikos}.

By substituting $C=3$, we obtain the two results of \cite{GM00} mentioned above.

In Section~\ref{sec:grunbaum}, we present a generalization of a classical 
result of Gr\"{u}nbaum \cite{Gru60}, according to which any halfspace 
containing the center of mass of a convex body contains at least $1/e$ of its 
volume. In Section~\ref{sec:epsilonnet}, we give a specific form of the 
$\varepsilon$-net theorem, a result from combinatorics obtained by Haussler and 
Welzl \cite{HW87} building on works of Vapnik and Chervonenkis \cite{VC68}, and 
then refined by Koml\'os, Pach and Woeginger \cite{KPW92}.
Finally, in Section~\ref{sec:proofofapproxthm}, we combine these two to obtain 
Theorem~\ref{thm:approxgeneral}.

\section{Convexity: A stability version of a theorem of 
Gr\"unbaum}\label{sec:grunbaum}

\emph{Gr\"unbaum's theorem} \cite{Gru60} states that for any centered convex 
body 
$K$ in $\Red$, and any half-space $F_0$ that contains the origin we have 
\begin{equation}\label{eq:grunbaum}
\vol{K}/e \leq \vol{K\cap F_0},
%\leq \vol{K}(1-1/e)
\end{equation}
where $\vol{\cdot}$ denotes volume.

We say that a half-space $F$ \emph{supports $K$ from outside} if the boundary 
of the half-space intersects $\bd K$, but $F$ does not intersect the interior 
of $K$. Lemma~\ref{lem:grunbaumstable}, is a stability version of 
Gr\"unbaum's theorem.

\begin{lem}\label{lem:grunbaumstable}
 Let $K$ be a convex body in $\Red$ with centroid at the origin. Let 
$0<\vartheta<1$, and $F$ be a half-space that supports $\vartheta K$ from 
outside. 
Then
\begin{equation}\label{eq:grunbaumstable}
\vol{K}\frac{(1-\vartheta)^d}{e}
\leq
 \vol{K\cap F}.
\end{equation}
\end{lem}

\begin{proof}[Proof of Lemma~\ref{lem:grunbaumstable}]
Let $F_0$ be a translate of $F$ containing $o$ on its boundary, and let $F_1$ 
be a translate of $F$ that supports $K$ from outside. Finally, let $p\in\bd 
F_1\cap K$. Then $\vartheta p+(1-\vartheta)(K\cap F_0)$ (that is, the 
homothetic copy of $K\cap F_0$ with homothety center $p$ and ratio 
$1-\vartheta$) is in $F$. Its volume is $(1-\vartheta)^d\vol{K\cap F_0}$, which 
by \eqref{eq:grunbaum}, is at least $(1-\vartheta)^d\vol{K}/e$, finishing the 
proof.
\end{proof}

\section{Combinatorics: The \texorpdfstring{$\varepsilon$}{epsilon}-net Theorem 
of Haussler and Welzl}\label{sec:epsilonnet}

\begin{defn}\label{defn:vcdim}
Let $\FF$ be a family of subsets of some set $U$. 
The \emph{Vapnik--Chervonenkis dimension} (\emph{VC-dimension}, in short) 
of 
$\FF$ is the maximal cardinality of a subset $V$ of $U$ such that $V$ is 
shattered by $\FF$, that is, $\{F\cap V\st F\in\FF\}=2^V$.

A \emph{transversal} of the set family $\FF$ is a subset $Q$ of $U$ that 
intersects each member of $\FF$.
\end{defn}

It is well known that if $U$ is any subset of $\Red$, and $\FF$ is a family of 
half-spaces of $\Red$, then the VC-dimension of $\FF$ is at most $d+1$.

The $\varepsilon$-net Theorem was first proved by 
Vapnik and Chervonenkis \cite{VC68}, then refined and applied in the geometric 
settings by Haussler and Welzl \cite{HW87}, and further improved by Koml\'{o}s, 
Pach and Woeginger \cite{KPW92}. We restate Theorem~3.1 of \cite{KPW92}.

\begin{lem}[$\varepsilon$-net Theorem]\label{lem:epsilonnet}
Let $0<\varepsilon<1/e$, and let $D$ be a positive integer. Let 
$\FF$ be a family of some measurable subsets of a probability space $(U, 
\mu)$, where the probability of each member $F$ of $\FF$ is 
$\mu(F)\geq\varepsilon$. Assume that the VC-dimension of $\FF$ is at most $D$. 
Let $t$ be 
\[
 t:=\left\lceil
 C\frac{D}{\varepsilon}\ln\frac{1}{\varepsilon}
 \right\rceil,
\]
where $C\geq 2$ is such that
\begin{equation*}
C^2\varepsilon^{C-2}\leq \frac{(\delta/4)^{1/D}}{e^2}.
\end{equation*}

Choose $t$ elements $X_1,\ldots,X_t$ of $V$ randomly, independently according 
to $\mu$. 
Then $\{X_1,\ldots,X_t\}$ is a transversal of $\FF$ with probability at 
least $1-\delta$.
\end{lem}

\begin{proof}
We give an outline of the last (routine, computational) part of the proof to 
obtain an explicit bound on the probability as stated in our lemma.

Let $E$ be the bad event, that is, when $\{X_1,\ldots,X_t\}$ is not a 
transversal of $\FF$.
At the end of the proof presented in \cite{PaAg95} (Theorem~15.5 therein), it 
is 
obtained that for any integer $T$ which is larger than $t$, we have
\[
 \prob(E)<
 2\sum_{i=0}^D \binom{T}{i} \left(1-\frac{t}{T}\right)^{(T-t)\varepsilon-1}.
\]
With the choice of $T=\frac{\varepsilon t^2}{D}$, using $\sum_{i=0}^D 
\binom{T}{i}\leq \left(\frac{eT}{D}\right)^D$, 
we obtain that if
$
\frac{e^2\varepsilon t^2}{D^2} e^{-\varepsilon t/D}< (\delta/4)^{1/D}
$
holds, then $ \prob(E)<\delta$, completing the proof of 
Lemma~\ref{lem:epsilonnet}.
\end{proof}

For more on the theory of $\varepsilon$-nets, see 
\cites{PaAg95, Ma02, AlSp16, MuVa17}.

\section{Proof of Theorem~\ref{thm:approxgeneral}}\label{sec:proofofapproxthm}
  
\begin{proof}[Proof of Theorem~\ref{thm:approxgeneral}]
 We consider the following set system on the base set $K$:
 \begin{equation*}
 \FF:=\{K\cap F\st F \mbox{ is a half space that supports } \vartheta K \mbox{ 
from outside}\}.
 \end{equation*}
Clearly, the VC-dimension of $\FF$ is at most $D:=d+1$. Let $\mu$ be the 
Lebesgue 
measure 
restricted to $K$, and assume that $\vol{K}=1$, that is, that $\mu$ is a 
probability measure. By \eqref{eq:grunbaumstable}, we have that each set in 
$\FF$ is of measure at least $\varepsilon:=\frac{(1-\vartheta)^d}{e}$. 
Lemma~\ref{lem:epsilonnet} yields that 
if we choose $t$ points of $K$ 
independently with respect to $\mu$ (that is, uniformly), then with probability 
at least $1-\delta$, we obtain a set $Q\subseteq K$ that intersects every 
member of $\FF$. The latter is equivalent to $\vartheta K\subseteq\conv{Q}$, 
completing the proof.
\end{proof}

\section*{Acknowledgements}
The author thanks Nabil Mustafa for enlightening conversations on the 
$\epsilon$-net theorem and topics around it.

The research was partially supported by the
National Research, Development and Innovation Office (NKFIH) grant 
NKFI-K119670 and by the J\'anos Bolyai Research Scholarship 
of the Hungarian Academy of Sciences. Part of the work was carried out during 
a stay at EPFL, Lausanne at J{\'a}nos Pach's Chair of Discrete and 
Computational 
Geometry supported by the Swiss National Science Foundation Grants 
200020-162884 and 200021-165977.

\bibliographystyle{amsalpha}
\bibliography{biblio}
\end{document}